\documentclass[10pt]{amsart}

\usepackage{amssymb,amsmath,amsthm}
\usepackage[arc,all]{xy}
\usepackage{eucal}

\usepackage{color}

\theoremstyle{plain}

\newtheorem{prop}[subsection]{Proposition}
\newtheorem{thm}[subsection]{Theorem}

\theoremstyle{definition}
\newtheorem{defn}[subsection]{Definition}

\theoremstyle{remark}
\newtheorem{rem}[subsection]{Remark}

\newcommand{\hh}{{ \mathsf{h} }}

\newcommand{\ZZ}{{ \mathbb{Z} }}

\newcommand{\sSet}{{ \mathsf{sSet} }}

\newcommand{\Mod}{{ \mathsf{Mod} }}
\newcommand{\ModR}{{ \mathsf{Mod}_\capR }}

\newcommand{\Spectra}{{ \mathsf{Sp}^\Sigma }}

\newcommand{\M}{{ \mathsf{M} }}

\newcommand{\Alg}{{ \mathsf{Alg} }}

\newcommand{\LL}{{ \mathsf{L} }}
\newcommand{\RR}{{ \mathsf{R} }}

\newcommand{\TQ}{{ \mathsf{TQ} }}

\newcommand{\AlgO}{{ \Alg_\capO }}

\newcommand{\capA}{{ \mathcal{A} }}

\newcommand{\capO}{{ \mathcal{O} }}

\newcommand{\capR}{{ \mathcal{R} }}

\newcommand{\id}{{ \mathrm{id} }}

\newcommand{\sslash}{{ /\!\!/ }}

\newcommand{\Smash}{{ \,\wedge\, }}

\newcommand{\tensor}{{ \otimes }}

\newcommand{\tensordot}{{ \dot{\tensor} }}

\newcommand{\wequiv}{{ \ \simeq \ }}

\newcommand{\Iso}{{  \ \cong \ }}
\newcommand{\Equal}{{ \ = \ }}

\newcommand{\function}[3]{{ {#1}\colon\thinspace{#2}\rightarrow{#3} }}

\DeclareMathOperator*{\hocolim}{hocolim}

\DeclareMathOperator*{\colim}{colim}

\DeclareMathOperator{\Hombold}{\mathbf{Hom}}

\title[Topological {Q}uillen localization of structured ring spectra]{Topological {Q}uillen localization\\ of structured ring spectra}

\author{John E. Harper}
\author{Yu Zhang}

\address{Department of Mathematics, The Ohio State University, Newark, 1179 University Dr, Newark, OH 43055, USA}

\email{harper.903@math.osu.edu}

\address{Department of Mathematics, The Ohio State University, 231 West 18th Ave, Columbus, OH 43210, USA}

\email{zhang.4841@osu.edu}

\begin{document}

\maketitle

\begin{abstract}
The aim of this short paper is two-fold: (i) to construct a $\TQ$-localization functor on algebras over a spectral operad $\capO$, in the case where no connectivity assumptions are made on the $\capO$-algebras, and (ii) more generally, to establish the associated $\TQ$-local homotopy theory as a left Bousfield localization of the usual model structure on $\capO$-algebras, which itself is almost never left proper, in general.  In the resulting $\TQ$-local homotopy theory, the ``weak equivalences'' are the $\TQ$-homology equivalences, where ``$\TQ$-homology'' is short for topological Quillen homology, which is also weakly equivalent to stabilization of $\capO$-algebras. More generally, we establish these results for $\TQ$-homology with coefficients in a spectral algebra $\capA$. A key observation, that goes back to the work of Goerss-Hopkins on moduli problems, is that the usual left properness assumption may be replaced with a strong cofibration condition in the desired subcell lifting arguments: Our main result is that the $\TQ$-local homotopy theory can be established (e.g., a semi-model structure in the sense of Goerss-Hopkins and Spitzweck, that is both cofibrantly generated and simplicial) by localizing with respect to a set of strong cofibrations that are $\TQ$-equivalences.
\end{abstract}

\section{Introduction}

In this paper we are working in the framework of algebras over an operad in symmetric spectra \cite{Hovey_Shipley_Smith, Schwede_book_project}, and more generally, in $\capR$-modules, where $\capO[0]=*$ (the trivial $\capR$-module); such $\capO$-algebras are non-unital. Here, $\capR$ is any commutative ring spectrum (i.e., any commutative monoid object in the category $(\Spectra,\tensor_S,S)$ of symmetric spectra) and we denote by $(\ModR,\wedge,\capR)$ the closed symmetric monoidal category of $\capR$-modules. 

Topological Quillen homology (or $\TQ$-homology) is the precise analog for $\capO$-algebras of singular homology for spaces, and is also weakly equivalent to stabilization of $\capO$-algebras \cite{Basterra_Mandell, Harper_Hess, Pereira_spectral_operad}. A useful starting point is \cite{Goerss_f2_algebras, Miller, Quillen}, together with \cite{Basterra, Basterra_Mandell, Basterra_Mandell_thh} and \cite{Ching_Harper_derived_Koszul_duality, Kuhn, Kuhn_adams_filtration, Kuhn_Pereira}; see also \cite{Ching_Harper, Ching_Harper_nilpotent_Whitehead, Fresse_lie_theory, Fresse, Harper_bar_constructions, Rezk}.

When $\TQ$-homology is iterated, built into a cosimplicial $\TQ$-resolution, and then glued all together with a homotopy limit, it gives the $\TQ$-completion \cite{Harper_Hess} (analogous to Bousfield-Kan \cite{Bousfield_Kan} completion for spaces). It is proved in \cite{Ching_Harper_derived_Koszul_duality} that $\TQ$-completion recovers the original $\capO$-algebra $X$, up to weak equivalence, provided that $X$ is 0-connected---in other words, 0-connected $\capO$-algebras are $\TQ$-complete; here, $\capO,\capR$ are assumed to be $(-1)$-connected.

So what about the larger class, for instance, of homotopy pro-nilpotent $\capO$-algebras---are they also $\TQ$-complete? This paper is a first step in attacking this problem; i.e., to construct the $\TQ$-localization as a ``better'' model than $\TQ$-completion for ``the part of an $\capO$-algebra that $\TQ$-homology sees''. $\TQ$-completion is known to only be ``the right model'' when the $\capO$-algebra $X$ is $\TQ$-good (i.e., when the comparison map from $X$ to its $\TQ$-completion is a $\TQ$-equivalence) analogous to the situation for spaces \cite{Bousfield_Kan}, but perhaps homotopy pro-nilpotent $\capO$-algebras are not $\TQ$-good, in general. So our attack on the problem is to first build (in this paper) $\TQ$-localization by establishing the $\TQ$-local homotopy theory for $\capO$-algebras (without any connectivity assumptions). Our motivation for constructing the $\TQ$-localization is that it always gives ``the right model'' for the part of the $\capO$-algebra X that $\TQ$-homology sees (at the expense of a much larger construction); just like Bousfield’s localization construction \cite{Bousfield_localization_spaces} for pointed spaces.

We follow closely the arguments in Bousfield \cite{Bousfield_localization_spaces}, Goerss-Jardine \cite{Goerss_Jardine}, and Jardine \cite{Jardine_local_homotopy_theory}; see also Dwyer \cite{Dwyer_localizations} for a useful introduction to these ideas, along with \cite{Bousfield_Kan, Farjoun_LNM, Hilton_Mislin_Roitberg, May_Ponto} in the context of spaces. To make the localization techniques work in the context of $\capO$-algebras, we exploit the cellular ideas in Hirschhorn \cite{Hirschhorn}. A potential wrinkle is the well-known failure (Remark \ref{rem:left_properness}), in general, of $\capO$-algebras to be left proper (e.g., associative ring spectra are not left proper); we show that exploiting an observation in Goerss-Hopkins \cite{Goerss_Hopkins_moduli_spaces, Goerss_Hopkins_moduli_problems} enables the desired topological Quillen localization to be constructed by localizing with respect to a particular set of \emph{strong} cofibrations that are $\TQ$-equivalences; the establishment of this $\TQ$-localization functor and the associated $\TQ$-local homotopy theory---as a semi-model structure that is cofibrantly generated and simplicial---are our main results; in other words, in this paper we establish the $\TQ$-local homotopy theory for $\capO$-algebras (Theorem \ref{thm:TQ_local_homotopy_theory}), essentially by re-examining ideas of Goerss-Hopkins \cite{Goerss_Hopkins, Goerss_Hopkins_moduli_problems} and Bousfield \cite{Bousfield_localization_spaces}, together with the technical work in Goerss-Jardine \cite{Goerss_Jardine} and Jardine \cite{Jardine_local_homotopy_theory}, in light of the cellular ideas and techniques in Hirschhorn \cite{Hirschhorn}.

As an application of the $\TQ$-local homotopy theory established here, together with the completion results in \cite{Ching_Harper_nilpotent_Whitehead}, it is shown in \cite{Zhang_homotopy_pro_nilpotent} that every homotopy pro-nilpotent $\capO$-algebra is $\TQ$-local; this improves the result in \cite{Ching_Harper_derived_Koszul_duality} that 0-connected $\capO$-algebras are $\TQ$-complete (assuming $\capO,\capR$ are $(-1)$-connected), to the much larger class of homotopy pro-nilpotent $\capO$-algebras, provided that one replaces ``$\TQ$-completion'' with ``$\TQ$-localization'', and is closely related to (and partially motivated by) a conjecture of Francis-Gaitsgory \cite[3.4.5]{Francis_Gaitsgory}. The $\TQ$-local homotopy theory developed here may also find potential applications for studying the closely related invariants in \cite{Galatius_Kupers_Randal_Williams, Heuts}.

To keep this paper appropriately concise, we freely use notation from \cite{Harper_Hess}.

\subsection*{Acknowledgments}

The authors would like to thank an anonymous referee for helpful comments and suggestions. The first author would like to thank Bill Dwyer, Emmanuel Farjoun, and Rick Jardine for useful discussions, at an early stage, on  localizations in homotopy theory. The authors would like to thank Matt Carr, Crichton Ogle, Nath Rao, and David White for helpful comments related to this work.

\section{$\TQ$-homology of an $\capO$-algebra with coefficients in $\capA$}

If $X$ is an $\capO$-algebra, then we may factor the map $*\rightarrow X$
\begin{align*}
  {*}\rightarrow\tilde{X}\xrightarrow{\wequiv}X
\end{align*}
as a cofibration followed by an acyclic fibration; we are using the positive flat stable model structure (see, for instance, \cite{Harper_Hess}). In particular, $\tilde{X}$ is a cofibrant replacement of $X$.

Consider the canonical map of operads $\function{f}{\capO}{\tau_1\capO}$ and any map $\function{\alpha}{\capO[1]}{\capA}$ of $\capR$-algebras. These maps induce adjunctions of the form
\begin{align}
\label{eq:basic_adjunctions}
\xymatrix{
  \AlgO\ar@<0.5ex>[r]^-{f_*} &
  \Alg_{\tau_1\capO}\Equal\Mod_{\capO[1]}\ar@<0.5ex>[l]^-{f^*}\ar@<0.5ex>[r]^-{\alpha_*} &
  \Mod_{\capA}\ar@<0.5ex>[l]^-{\alpha^*}
}
\end{align}
with left adjoints on top, where $f_*(X):=\tau_1\capO\circ_\capO(X)$ and $f^*$ denotes ``forgetting along $f$ of the left $\tau_1\capO$-action'', and similarly, $\alpha_*(Y):=\capA\Smash_{\capO[1]} Y$ and $\alpha^*$ denotes ``forgetting along $\alpha$ of the left $\capA$-action''; for short, we sometimes refer to $f^*$ and $\alpha^*$ as the indicated ``forgetful functors''. For notational convenience purposes, we denote by $Q:=\alpha_*f_*$ the composite of left adjoints in \eqref{eq:basic_adjunctions} and by $U:=f^*\alpha^*$ the composite of right adjoints in \eqref{eq:basic_adjunctions}. It follows that $(Q,U)$ fit into an adjunction of the form
\begin{align}
\label{eq:TQ_with_coefficients_adjunction}
\xymatrix{
  \AlgO\ar@<0.5ex>[r]^-{Q} &
  \Mod_\capA\ar@<0.5ex>[l]^-{U}
}
\end{align}
with left adjoint on top; here, $Q$ is for indecomposable ``quotient'' and $U$ is the indicated forgetful functor.

\begin{defn}
If $X$ is an $\capO$-algebra, then its \emph{$\TQ$-homology} is 
\begin{align*}
  \TQ(X):=\tau_1\capO\circ_\capO^\hh(X):=\RR f^*(\LL f_*(X))\wequiv\tau_1\capO\circ_\capO(\tilde{X})
\end{align*}
the $\capO$-algebra defined by the indicated composite of total right and left derived functors, 
and its \emph{$\TQ$-homology with coefficients in $\capA$}, is the $\capO$-algebra
\begin{align*}
  \TQ^\capA(X):=\RR U(\LL Q(X))\wequiv Q(\tilde{X})\Equal
  \capA\Smash_{\capO[1]}\bigl(\tau_1\capO\circ_\capO(\tilde{X})\bigr)
\end{align*}
In particular, if the algebra map $\alpha=\id$, then $\TQ^{\capO[1]}(X)\wequiv \TQ(X)$. Here, $\TQ$-homology is short for ``topological Quillen homology'' which is weakly equivalent to stabilization of $\capO$-algebras.
\end{defn}

\section{Detecting $\TQ^\capA$-local $\capO$-algebras}

\begin{defn}
\label{defn:strong_cofibrations}
A map $\function{i}{A}{B}$ of $\capO$-algebras is a \emph{strong cofibration} if it is a cofibration between cofibrant objects in $\AlgO$.
\end{defn}

\begin{defn}
\label{defn:TQ_local_O_algebras}
Let $X$ be an $\capO$-algebra. We say that $X$ is \emph{$\TQ^\capA$-local} if (i) $X$ is fibrant in $\AlgO$ and (ii) every strong cofibration $A\rightarrow B$ that induces a weak equivalence $\TQ^\capA(A)\wequiv\TQ^\capA(B)$ on $\TQ^\capA$-homology, induces a weak equivalence
\begin{align}
\label{eq:induced_map_on_mapping_complexes}
  \Hombold(A,X)\xleftarrow{\wequiv}\Hombold(B,X)
\end{align}
on mapping spaces in $\sSet$.
\end{defn}

\begin{rem}
The intuition here is that the derived space of maps into a $\TQ^\capA$-local $\capO$-algebra cannot distinguish between $\TQ^\capA$-equivalent $\capO$-algebras (Proposition \ref{prop:mapping_into_TQ_local_O_algebras}), up to weak equivalence.
\end{rem}

Evaluating the map \eqref{eq:induced_map_on_mapping_complexes} at level 0 gives a surjection
\begin{align*}
  \hom(A,X)\leftarrow\hom(B,X)
\end{align*}
of sets, since acyclic fibrations in $\sSet$ are necessarily levelwise surjections. This suggests that $\TQ^\capA$-local $\capO$-algebras $X$ might be detected by a right lifting property and motivates the following classes of maps (Proposition \ref{prop:detecting_TQ_local_O_algebras_part_1}); compare with Bousfield \cite{Bousfield_localization_spaces}.

\begin{defn}[$\TQ^\capA$-local homotopy theory: Classes of maps]
\label{defn:classes_of_maps_TQ_local_homotopy_theory}
A map $\function{f}{X}{Y}$ of $\capO$-algebras is
\begin{itemize}
\item[(i)] a \emph{$\TQ^\capA$-equivalence} if it induces a weak equivalence $\TQ^\capA(X)\wequiv\TQ^\capA(Y)$
\item[(ii)] a \emph{$\TQ^\capA$-cofibration} if it is a cofibration in $\AlgO$
\item[(iii)] a \emph{$\TQ^\capA$-fibration} if it has the right lifting property with respect to every cofibration that is a $\TQ^\capA$-equivalence
\item[(iv)] a \emph{weak $\TQ^\capA$-fibration} (or \emph{$\TQ^\capA$-injective fibration; see Jardine \cite{Jardine_local_homotopy_theory}}) if it has the right lifting property with respect to every strong cofibration that is a $\TQ^\capA$-equivalence
\end{itemize}
A cofibration (resp. strong cofibration) is called \emph{$\TQ^\capA$-acyclic} if it is also a $\TQ^\capA$-equivalence. Similarly, a $\TQ^\capA$-fibration (resp. weak $\TQ^\capA$-fibration) is called \emph{$\TQ^\capA$-acyclic} if it is also a $\TQ^\capA$-equivalence.
\end{defn}

\begin{rem}
\label{rem:left_properness}
The additional class of maps (iv) naturally arises in the $\TQ^\capA$-local homotopy theory established below (Theorem \ref{thm:TQ_local_homotopy_theory}); this is a consequence of the fact that the model structure on $\AlgO$ is almost never left proper, in general (e.g., associative ring spectra are not left proper); see, for instance, Goerss-Hopkins \cite[2.3]{Goerss_Hopkins_moduli_spaces}. In the very few special cases where it happens that $\AlgO$ is left proper (e.g., commutative ring spectra are left proper), then the class of weak $\TQ^\capA$-fibrations will be identical to the class of $\TQ^\capA$-fibrations.
\end{rem}

\begin{prop}
\label{prop:relations_between_classes_of_maps}
The following implications are satisfied
\begin{align*}
  \text{strong cofibration}\ &\Longrightarrow\ 
  \text{cofibration}\\
  \text{weak equivalence}\ &\Longrightarrow\ 
  \text{$\TQ^\capA$-equivalence}\\
  \text{$\TQ^\capA$-fibration}\ &\Longrightarrow\ 
  \text{weak $\TQ^\capA$-fibration}\ \Longrightarrow\ 
  \text{fibration}
\end{align*}
for maps of $\capO$-algebras.
\end{prop}

\begin{proof}
The first implication is immediate and the second is because $\TQ^\capA$ preserves weak equivalences, by construction. The third implication is because the class of $\TQ^\capA$-acyclic cofibrations contains the class of $\TQ^\capA$-acyclic strong cofibrations. For the last implication, recall that a map is a fibration in $\AlgO$ if and only if it has the right lifting property with respect to the set of generating acyclic cofibrations. Since the generating acyclic cofibrations have cofibrant domains \cite{Shipley_commutative_ring_spectra}, they are contained in the class of strong cofibrations that are weak equivalences, and hence they are contained in the class of $\TQ^\capA$-acyclic strong cofibrations. It follows immediately that every weak $\TQ^\capA$-fibration is a fibration.
\end{proof}

\begin{prop}
\label{prop:mapping_into_TQ_local_O_algebras}
Let $X$ be a fibrant $\capO$-algebra. Then $X$ is $\TQ^\capA$-local if and only if every map $\function{f}{A}{B}$ between cofibrant $\capO$-algebras that is a $\TQ^\capA$-equivalence induces a weak equivalence \eqref{eq:induced_map_on_mapping_complexes} on mapping spaces.
\end{prop}

\begin{proof}
It suffices to verify the ``only if'' direction. Consider any map $\function{f}{A}{B}$ between cofibrant $\capO$-algebras that is a $\TQ^\capA$-equivalence. Factor $f$ as a cofibration $i$ followed by an acyclic fibration $p$ in $\AlgO$. Since $f$ is a $\TQ^\capA$-equivalence and $p$ is a weak equivalence, it follows that $i$ is a $\TQ^\capA$-equivalence.  The left-hand commutative diagram induces 
\begin{align*}
\xymatrix{
  A\ar[r]^-{f}\ar[d]^-{i} & B\\
  B'\ar@/_0.5pc/[ur]_-{p}
}\quad\quad
\xymatrix{
  \Hombold(A,X) & \Hombold(B,X)\ar[l]_-{(*)}
  \ar@/^0.5pc/[dl]^-{(\#)}\\
  \Hombold(B',X)\ar[u]^-{(**)}
}
\end{align*}
the right-hand commutative diagram. Since $p$ is a weak equivalence between cofibrant objects and $X$ is fibrant, we know that $(\#)$ is a weak equivalence, hence $(*)$ is a weak equivalence if and only if $(**)$ is a weak equivalence. Since $i$ is a strong cofibration, by construction, this completes the proof.
\end{proof}

\begin{prop}
\label{prop:TQ_acyclic_TQ_fibrations}
Consider any map $\function{f}{X}{Y}$ of $\capO$-algebras. Then the following are equivalent:
\begin{itemize}
\item[(i)] $f$ is a weak $\TQ^\capA$-fibration and $\TQ^\capA$-equivalence
\item[(ii)] $f$ is a $\TQ^\capA$-fibration and $\TQ^\capA$-equivalence
\item[(iii)] $f$ is a fibration and weak equivalence
\end{itemize}
\end{prop}

\begin{proof}
We know that $\text{(iii)}\Rightarrow\text{(ii)}$ because weak equivalences are $\TQ^\capA$-equivalences (Proposition \ref{prop:relations_between_classes_of_maps}) and acyclic fibrations have the right lifting property with respect to cofibrations. Note that  $\text{(ii)}\Rightarrow\text{(i)}$ by Proposition \ref{prop:relations_between_classes_of_maps}, hence it suffices to verify the implication $\text{(i)}\Rightarrow\text{(iii)}$. Suppose $f$ is a weak $\TQ^\capA$-fibration and $\TQ^\capA$-equivalence; let's verify that $f$ is an acyclic fibration. Since every generating cofibration for $\AlgO$ is a strong cofibration, it suffices to verify that $f$ has the right lifting property with respect to strong cofibrations. Let $\function{i}{A}{B}$ be a strong cofibration.  We want to verify that the left-hand solid commutative diagram 
\begin{align*}
\xymatrix{
  A\ar[r]^-{g}\ar[d]_-{i} & X\ar[d]^-{f}\\
  B\ar[r]_-{h} & Y
}\quad\quad
\xymatrix{
  A\ar[r]^-{g'}\ar[d]_-{i} & \tilde{X}\ar[r]^-{g''}\ar[dr]_(0.4){f'} & X\ar[d]^-{f}\\
  B\ar[rr]_-{h}\ar@{.>}[ur]^(0.45){\xi} && Y
}
\end{align*}
in $\AlgO$ has a lift. We factor $g$ as a cofibration followed by an acyclic fibration   
$
  A\xrightarrow{g'}\tilde{X}\xrightarrow{g''} X
$
in $\AlgO$. It follows easily that the composite $f':=fg''$ is a weak $\TQ^\capA$-fibration and $\TQ^\capA$-equivalence with cofibrant domain. To verify that the desired lift $\xi$ exists, it is enough to check that $f'$ is an acyclic fibration.

We factor $f'$ as a cofibration followed by an acyclic fibration
$
  \tilde{X}\xrightarrow{j}\tilde{Y}\xrightarrow{p} Y
$
in $\AlgO$, and since $f',p$ are $\TQ^\capA$-equivalences, it follows that $j$ is a $\TQ^\capA$-equivalence. Hence $j$ is a $\TQ^\capA$-acyclic  strong cofibration and the left-hand solid commutative diagram
\begin{align}
\label{eq:retract_argument_TQ}
\xymatrix{
  \tilde{X}\ar@{=}[r]\ar[d]_-{j} & \tilde{X}\ar[d]^-{f'}\\
  \tilde{Y}\ar[r]_-{p}\ar@{.>}[ur]^-{\eta} & Y
}\quad\quad
\xymatrix{
  \tilde{X}\ar[d]_-{f'}\ar[r]^-{j} & \tilde{Y}\ar[d]_-{p}\ar@{.>}[r]^-{\eta} & \tilde{X}\ar[d]_-{f'}\\
  Y\ar@{=}[r] & Y\ar@{=}[r] & Y
}
\end{align}
has a lift $\eta$. It follows that the right-hand diagram commutes with upper horizontal composite the identity map; in particular, $f'$ is a retract of $p$. Therefore $f'$ is an acyclic fibration which completes the proof. 
\end{proof}

The following is proved, for instance, in \cite[7.6]{Ching_Harper_derived_Koszul_duality}.

\begin{prop}
\label{prop:tensordot_commutes_with_Q}
If $A$ is an $\capO$-algebra and $K\in\sSet$, then there are isomorphisms $Q(A\tensordot K)\Iso Q(A)\tensordot K$ in $\Mod_\capA$, natural in $A,K$.
\end{prop}

\begin{prop}
\label{prop:TQ_acyclic_strong_cofibrations_play_nicely_with_tensordot}
If $\function{j}{A}{B}$ is a strong cofibration of $\capO$-algebras and $\function{i}{K}{L}$ is a cofibration in $\sSet$, then the pushout corner map
\begin{align*}
  A\tensordot L\amalg_{A\tensordot K}B\tensordot K
  \rightarrow B\tensordot L
\end{align*}
in $\AlgO$ is a strong cofibration that is a $\TQ^\capA$-equivalence if $j$ is a $\TQ^\capA$-equivalence.
\end{prop}

\begin{proof}
We know that the pushout corner map is a strong cofibration by the simplicial model structure on $\AlgO$ (see, for instance, \cite{Harper_Hess}), hence it suffices to verify that $Q$ applied to this map is a weak equivalence. Since $Q$ is a left Quillen functor, it follows  that the pushout corner map
\begin{align*}
  Q(A)\tensordot L\amalg_{Q(A)\tensordot K}Q(B)\tensordot K
  \rightarrow Q(B)\tensordot L
\end{align*}
is a cofibration that is a weak equivalence if $Q(A)\rightarrow Q(B)$ is a weak equivalence, and Proposition \ref{prop:tensordot_commutes_with_Q} completes the proof.
\end{proof}

\begin{prop}
\label{prop:pullback_corner_map}
If $\function{j}{A}{B}$ is a strong cofibration and $\function{p}{X}{Y}$ is a weak $\TQ^\capA$-fibration of $\capO$-algebras, then the pullback corner map
\begin{align}
\label{eq:pullback_corner_map_sSet}
  \Hombold(B,X)\rightarrow\Hombold(A,X)\times_{\Hombold(A,Y)}\Hombold(B,Y)
\end{align}
in $\sSet$ is a fibration that is an acyclic fibration if either $j$ or $p$  is a $\TQ^\capA$-equivalence.
\end{prop}

\begin{proof}
Suppose $j$ is a $\TQ^\capA$-acyclic strong cofibration and $p$ is a weak $\TQ^\capA$-fibration. Consider any cofibration $\function{i}{K}{L}$ in $\sSet$. We want to show that the pullback corner map \eqref{eq:pullback_corner_map_sSet} satisfies the right lifting property with respect to $i$.
\begin{align}
\label{eq:lifting_diagram_for_simplicial_structure_proof}
\xymatrix{
  K\ar[d]\ar[r] & \Hombold(B,X)\ar[d]\\
  L\ar[r]\ar@{.>}[ur] & \Hombold(A,X)\times_{\Hombold(A,Y)}\Hombold(B,Y)
}\quad\quad
\xymatrix{
  A\tensordot L \amalg_{A\tensordot K} B\tensordot K\ar[d]_-{(*)}\ar[r] & 
  X\ar[d]\\
  B\tensordot L\ar[r]\ar@{.>}[ur] & Y
}
\end{align}
The left-hand solid commutative diagram has a lift if and only if the corresponding right-hand solid commutative diagram has a lift. Noting that $(*)$ is a $\TQ^\capA$-acyclic strong cofibration (Proposition \ref{prop:TQ_acyclic_strong_cofibrations_play_nicely_with_tensordot}) completes the proof of this case. Suppose $j$ is a strong cofibration and $p$ is a weak $\TQ^\capA$-fibration. Consider any acyclic cofibration $\function{i}{K}{L}$ in $\sSet$. We want to show that the pullback corner map \eqref{eq:pullback_corner_map_sSet} satisfies the right lifting property with respect to $i$. The left-hand solid commutative diagram in \eqref{eq:lifting_diagram_for_simplicial_structure_proof} has a lift if and only if the corresponding right-hand solid commutative diagram has a lift. Noting that $p$ is a fibration (Proposition \ref{prop:relations_between_classes_of_maps}) and $(*)$ is an acyclic cofibration (see, for instance, \cite[Section 6]{Harper_Hess}) completes the proof of this case. The case where $j$ is a strong cofibration and $p$ is a $\TQ^\capA$-acyclic weak $\TQ^\capA$-fibration is similar; this is because $p$ is an acyclic fibration (Proposition \ref{prop:TQ_acyclic_TQ_fibrations}).
\end{proof}

\begin{prop}[Detecting $\TQ^\capA$-local $\capO$-algebras: Part 1]
\label{prop:detecting_TQ_local_O_algebras_part_1}
Let $X$ be a fibrant $\capO$-algebra. Then $X$ is $\TQ^\capA$-local if and only if $X\rightarrow *$ satisfies the right lifting property with respect to every $\TQ^\capA$-acyclic strong cofibration $A\rightarrow B$ of $\capO$-algebras (i.e., if and only if $X\rightarrow *$ is a weak $\TQ^\capA$-fibration).
\end{prop}

\begin{proof}
Suppose $X$ is $\TQ^\capA$-local and let $\function{i}{A}{B}$ be a $\TQ^\capA$-acyclic strong cofibration. Let's verify that $X\rightarrow *$ satisfies the right lifting property with respect to $i$. We know that the induced map of simplicial sets \eqref{eq:induced_map_on_mapping_complexes} is an acyclic fibration, hence evaluating the induced map \eqref{eq:induced_map_on_mapping_complexes} at level 0 gives a surjection
\begin{align*}
 \hom(A,X)\leftarrow\hom(B,X)
\end{align*}
of sets, which verifies the desired lift exists. Conversely, consider any $\TQ^\capA$-acyclic strong cofibration $A\rightarrow B$ of $\capO$-algebras. Let's verify that the induced map \eqref{eq:induced_map_on_mapping_complexes} is an acyclic fibration. It suffices to verify the right lifting property with respect to any generating cofibration $\partial\Delta[n]\rightarrow\Delta[n]$ in $\sSet$. Consider any left-hand solid commutative diagram of the form
\begin{align*}
\xymatrix{
  \partial\Delta[n]\ar[d]\ar[r] & \Hombold(B,X)\ar[d]\\
  \Delta[n]\ar[r]\ar@{.>}[ur] & \Hombold(A,X)
}\quad\quad
\xymatrix{
  A\tensordot\Delta[n]\coprod_{A\tensordot\partial\Delta[n]} B\tensordot\partial\Delta[n]\ar[d]^-{(*)}\ar[r] & X\ar[d]\\
  B\tensordot\Delta[n]\ar[r]\ar@{.>}[ur] & {*}
}
\end{align*}
in $\sSet$. Then the left-hand lift exists in $\sSet$ if and only if the corresponding right-hand lift exists in $\AlgO$. The map $(*)$ is a $\TQ^\capA$-acyclic strong cofibration by Proposition \ref{prop:TQ_acyclic_strong_cofibrations_play_nicely_with_tensordot}, hence, by assumption, the lift in the right-hand diagram exists, which completes the proof.
\end{proof}

\begin{rem}
\label{rem:dropping_fibrancy_assumption}
Since the generating acyclic cofibrations in $\AlgO$ have cofibrant domains, the fibrancy assumption on $X$ in Proposition \ref{prop:detecting_TQ_local_O_algebras_part_1} could be dropped; we keep it in, however, to motivate later closely related statements (Propositions \ref{prop:detecting_TQ_local_O_algebras_part_2} and \ref{prop:detecting_TQ_local_O_algebras_part_3}).
\end{rem}

\section{Cell $\capO$-algebras and the subcell lifting property}

Suppose we start with an $\capO$-algebra $A$. It may not be cofibrant, so we can run the small object argument with respect to the set of generating cofibrations in $\AlgO$ for the map $*\rightarrow A$. This gives a factorization in $\AlgO$ as $*\rightarrow \tilde{A}\rightarrow A$ a cofibration followed by an acyclic fibration. In particular, this construction builds $\tilde{A}$ by attaching cells; we would like to think of $\tilde{A}$ as a ``cell $\capO$-algebra'', and we will want to work with a useful notion of ``subcell $\capO$-algebra'' obtained by only attaching a subset of the cells above. Since every $\capO$-algebra can be replaced by such a cell $\capO$-algebra, up to weak equivalence, the idea is that this should provide a convenient class of $\capO$-algebras to reduce to when constructing the $\TQ^\capA$-localization functor; this reduction strategy---to work with cellular objects---is one of the main themes in Hirschhorn \cite{Hirschhorn}, and it plays a key role in this paper. The first step is to recall the generating cofibrations for $\AlgO$ and to make these cellular ideas more precise in the particular context of $\capO$-algebras needed for this paper.

Recall from \cite[7.10]{Harper_Hess} that the generating cofibrations for the positive flat stable model structure on $\capR$-modules is given by the set of maps of the form
\begin{align*}
\xymatrix{ 
  \capR\tensor G^H_m \partial\Delta[k]_+\ar[r]^-{i_m^{H,k}}&
  \capR\tensor G^H_m\Delta[k]_+
}\quad\quad
  (m\geq 1,\ k\geq 0,\ H\subset \Sigma_m \ \text{subgroup})
\end{align*} 
in $\capR$-modules. For ease of notational purposes, it will be convenient to denote this set of maps using the more concise notation
\begin{align*}
\xymatrix{ 
  S_m^{H,k}\ar[r]^-{i_m^{H,k}}&
  D_m^{H,k}
}\quad\quad
  (m\geq 1,\ k\geq 0,\ H\subset \Sigma_m \ \text{subgroup})
\end{align*}
where $S_m^{H,k}$ are $D_m^{H,k}$ are intended to remind the reader of ``sphere'' and ``disk'', respectively. In terms of this notation, recall from \cite[7.15]{Harper_Hess} that the generating cofibrations for the positive flat stable model structure on $\capO$-algebras is given by the set of maps of the form
\begin{align}
\label{eq:generating_cofibrations_for_O_algebras}
\xymatrix{ 
  \capO\circ(S_m^{H,k})\ar[rr]^-{\id\circ(i_m^{H,k})}&&
  \capO\circ(D_m^{H,k})
}\quad\quad
  (m\geq 1,\ k\geq 0,\ H\subset \Sigma_m \ \text{subgroup})
\end{align}
in $\capO$-algebras. 

Definitions \ref{defn:relative_cell_O_algebra}--\ref{defn:subcell_O_algebra} below appear in Hirschhorn \cite[10.5.8, 10.6]{Hirschhorn} in the more general context of cellular model categories; we have tailored the definitions to exactly what is needed for  this paper; i.e., in the context of $\capO$-algebras.

\begin{defn}
\label{defn:relative_cell_O_algebra}
A map $\function{\alpha}{W}{Z}$ in $\AlgO$ is a \emph{relative cell $\capO$-algebra} if it can be constructed as a transfinite composition of maps of the form
\begin{align*}
  W=Z_0\rightarrow Z_1\rightarrow Z_2\rightarrow \dots\rightarrow Z_\infty:=\colim_n Z_n\Iso Z
\end{align*}
such that each map $Z_n\rightarrow Z_{n+1}$ is built from a pushout diagram of the form
\begin{align}
\label{eq:presentation_of_relative_cell_O_algebra}
\xymatrix{
  \coprod_{i\in I_n}\capO\circ(S_{m_i}^{H_i,k_i})\ar[d]_-{\amalg_{i\in I_n}\id\circ(i_{m_i}^{H_i,k_i})}\ar[r]^-{(*)} & Z_n\ar[d]\\
  \coprod_{i\in I_n}\capO\circ(D_{m_i}^{H_i,k_i})\ar[r] & Z_{n+1}
}
\end{align}
in $\AlgO$, for each $n\geq 0$. A choice of such a transfinite composition of pushouts is a \emph{presentation} of $\function{\alpha}{W}{Z}$ as a relative cell $\capO$-algebra. With respect to such a presentation, the \emph{set of cells} in $\alpha$ is the set $\sqcup_{n\geq 0}I_n$ and the \emph{number of cells} in $\alpha$ is the cardinality of its set of cells; here, $\sqcup$ denotes disjoint union of sets.
\end{defn}

\begin{rem}
We often drop explicit mention of the choice of presentation of a relative cell $\capO$-algebra, for ease of reading purposes, when no confusion can result.
\end{rem}

\begin{defn}
An $\capO$-algebra $Z$ is a \emph{cell $\capO$-algebra} if $*\rightarrow Z$ is a relative cell $\capO$-algebra. The \emph{number of cells} in $Z$, denoted $\# Z$, is the number of cells in $*\rightarrow Z$ (with respect to a choice of presentation of $*\rightarrow Z$).
\end{defn}

\begin{defn}
\label{defn:subcell_O_algebra}
Let $Z$ be a cell $\capO$-algebra. A \emph{subcell $\capO$-algebra} of $Z$ is a cell $\capO$-algebra $Y$ built by a subset of cells in $Z$ (with respect to a choice of presentation of $*\rightarrow Z$). More precisely, $Y\subset Z$ is a subcell $\capO$-algebra if $*\rightarrow Y$ can be constructed as a transfinite composition of maps of the form
\begin{align*}
  *=Y_0\rightarrow Y_1\rightarrow Y_2\rightarrow \dots\rightarrow Y_\infty:=\colim_n Y_n\Iso Y
\end{align*}
such that each map $Y_n\rightarrow Y_{n+1}$ is built from a pushout diagram of the form
\begin{align*}
\xymatrix{
  \coprod_{j\in J_n}\capO\circ(S_{m_j}^{H_j,k_j})\ar[d]_-{\amalg_{j\in J_n}\id\circ(i_{m_j}^{H_j,k_j})}\ar[r]^-{(**)} & Y_n\ar[d]\\
  \coprod_{j\in J_n}\capO\circ(D_{m_j}^{H_j,k_j})\ar[r] & Y_{n+1}
}
\end{align*}
in $\AlgO$, where $J_n\subset I_n$ and the attaching map $(**)$ is the restriction of the corresponding attaching map $(*)$ in \eqref{eq:presentation_of_relative_cell_O_algebra} (taking $W=*$), for each $n\geq 0$.
\end{defn}

\begin{defn}
Let $Z$ be a cell $\capO$-algebra. A subcell $\capO$-algebra $Y\subset Z$ is \emph{finite} if $\# Y$ is finite (with respect to a choice of presentation of $*\rightarrow Z$); in this case we say that $Y$ has finitely many cells.
\end{defn}

\begin{rem}
Let $Z$ be a cell $\capO$-algebra. A subcell $\capO$-algebra $Y\subset Z$ can be described by giving a compatible collection of subsets $J_n\subset I_n$, $n\geq 0$, (with respect to a choice of presentation for $*\rightarrow Z$); here, \emph{compatible} means that the corresponding attaching maps are well-defined. It follows that the resulting subcell $\capO$-algebra inclusion $Y\subset Z$ can be constructed stage-by-stage
\begin{align*}
\xymatrix{
  {*}=Y_0\ar@{=}@<-3ex>[d]\ar[r] & Y_1\ar[d]\ar[r] & Y_2\ar[d]\ar[r] & \dots\ar[r] & Y_\infty\ar[d]\ar[r]^-{\Iso} & Y\ar[d]\\
  {*}=Z_0\ar[r] & Z_1\ar[r] & Z_2\ar[r] & \dots\ar[r] & Z_\infty\ar[r]^-{\Iso} & Z 
}
\end{align*}
as the indicated colimit.
\end{rem}

\begin{prop}
\label{prop:pushout_diagram_of_subcell_O_algebras}
Let $Z$ be a cell $\capO$-algebra. If $A\subset Z$ and $B\subset Z$ are subcell $\capO$-algebras, then there is a pushout diagram of the form
\begin{align}
\label{eq:pushout_diagram_of_subcell_O_algebras}
\xymatrix{
  A\cap B\ar[d]\ar[r] & A\ar[d]\\
  B\ar[r] & A\cup B
}
\end{align}
in $\AlgO$, which is also a pullback diagram, where the indicated arrows are subcell $\capO$-algebra inclusions.
\end{prop}

\begin{proof}
This is proved in Hirschhorn \cite[12.2.2]{Hirschhorn} in a more general context, but here is the basic idea: Consider $*\rightarrow Z$ with presentation as in \eqref{eq:presentation_of_relative_cell_O_algebra} (taking $W=*$). Suppose that $S_n\subset I_n$ and $T_n\subset I_n$, $n\geq 0$, correspond to the subcell $\capO$-algebras $A\subset Z$ and $B\subset Z$, respectively. Then it follows (by induction on $n$) that $S_n\cap T_n\subset I_n$ and $S_n\cup T_n\subset I_n$, $n\geq 0$, are compatible collections of subsets and taking $A\cap B\subset Z$ and $A\cup B\subset Z$ to be the corresponding subcell $\capO$-algebras, respectively, completes the proof. Here, we are using the fact that every cofibration of $\capO$-algebras is, in particular, a monomorphism of underlying symmetric spectra, and hence an effective monomorphism \cite[12.2]{Hirschhorn} of $\capO$-algebras.
\end{proof}

The following is proved in \cite[I.2.4, I.2.5]{Chacholski_Scherer}.
\begin{prop}
\label{prop:certain_pushouts_are_homotopy_pushouts}
Let $\M$ be a model category (see, for instance, \cite[3.3]{Dwyer_Spalinski}).
\begin{itemize}
\item[(a)] Consider any pushout diagram of the form
\begin{align*}
\xymatrix{
  A\ar[d]_-{i}\ar[r]^-{f} & B\ar[d] \\
  C\ar[r]^-{g} & D
}
\end{align*}
in $\M$, where $A,B,C$ are cofibrant and $i$ is a cofibration. If $f$ is a weak equivalence, then $g$ is a weak equivalence.
\item[(b)] Consider any commutative diagram of the form
\begin{align*}
\xymatrix{
  A_0\ar[d]_-{\simeq} & 
  A_1\ar[d]_-{\simeq}\ar[r]\ar[l] & 
  A_2\ar[d]_-{\simeq}\\
  B_0 & 
  B_1\ar[r]\ar[l] & 
  B_2
}
\end{align*}
in $\M$, where $A_i,B_i$ are cofibrant for each $0\leq i\leq 2$, the vertical maps are weak equivalences, and $A_0\leftarrow A_1$ is a cofibration. If either $B_0\leftarrow B_1$ or $B_1\rightarrow B_2$ is a cofibration, then the induced map
\begin{align*}
  A_0\amalg_{A_1}A_2\xrightarrow{\wequiv}
  B_0\amalg_{B_1}B_2
\end{align*}
is a weak equivalence.
\end{itemize}
\end{prop}

The following proposition, which is an exercise left to the reader, has been exploited, for instance, in \cite[2.1]{Bauer_Johnson_McCarthy} and \cite[13.2.1]{Hirschhorn}; it is closely related to the usual induced model structures on over-categories and under-categories; see, for instance,  \cite[3.10]{Dwyer_Spalinski}.

\begin{prop}[Factorization category of a map]
\label{prop:factorization_category_of_a_map}
Let $\M$ be a model category and $\function{z}{A}{Y}$ a map in $\M$. Denote by $\M(z)$ the category with objects the factorizations $\mathbf{X}\colon\thinspace A\rightarrow X\rightarrow Y$ of $z$ in $\M$ and morphisms $\function{\xi}{\mathbf{X}}{\mathbf{X'}}$ the commutative diagrams of the form
\begin{align*}
\xymatrix{
  \mathbf{X}:\ar@<-0.5ex>[d]^-{\xi} & 
  A\ar@{=}[d]\ar[r] & 
  X\ar[d]^-{\xi}\ar[r] & 
  Y\ar@{=}[d]\\
  \mathbf{X}': & 
  A\ar[r] & 
  X'\ar[r] & 
  Y
}
\end{align*}
in $\M$. Define a map $\function{\xi}{\mathbf{X}}{\mathbf{X}'}$ to be a weak equivalence (resp. fibration, resp. cofibration) if $\function{\xi}{X}{X'}$ is a weak equivalence (resp. fibration, resp. cofibration) in $\M$. With these three classes of maps, $\M(z)$ inherits a naturally occurring model structure from $\M$. Since the initial object (resp. terminal object) in $\M(z)$ has the form $A=A\xrightarrow{z} Y$ (resp. $A\xrightarrow{z}Y=Y$), it follows that $\mathbf{X}$ is cofibrant (resp. fibrant) if and only if $A\rightarrow X$ is a cofibration (resp. $X\rightarrow Y$ is a fibration) in $\M$.
\end{prop}

\begin{proof}
This appears in \cite[2.1]{Bauer_Johnson_McCarthy} and is closely related to \cite[3.10]{Dwyer_Spalinski} and \cite[II.2.8]{Quillen}.
\end{proof}

The following subcell lifting property can be thought of as an $\capO$-algebra analog of Hirschhorn \cite[13.2.1]{Hirschhorn} as a key step in  establishing localizations in left proper celluar model categories. One technical difficulty with Proposition \ref{prop:detecting_TQ_local_O_algebras_part_1} for detecting $\TQ^\capA$-local $\capO$-algebras is that it involves a lifting condition with respect to a collection of maps, instead of a set of maps. Proposition \ref{prop:subcell_lifting_property} provides our first reduction towards eventually refining the lifting criterion for $\TQ^\capA$-local $\capO$-algebras to a set of maps. Even though the left properness assumption in \cite[13.2.1]{Hirschhorn} is almost never satisfied by $\capO$-algebras, in general, a key observation, that goes back to the work of Goerss-Hopkins \cite[1.5]{Goerss_Hopkins_moduli_problems} on moduli problems, is that the subcell lifting argument only requires an appropriate pushout diagram to be a homotopy pushout diagram---this is ensured (Proposition \ref{prop:certain_pushouts_are_homotopy_pushouts}) by the strong cofibration condition in Proposition \ref{prop:subcell_lifting_property}.

\begin{prop}[Subcell lifting property]
\label{prop:subcell_lifting_property}
Let $\function{p}{X}{Y}$ be a fibration of $\capO$-algebras. Then the following are equivalent:
\begin{itemize}
\item[(a)] The map $p$ has the right lifting property with respect to every strong cofibration $A\rightarrow B$ of $\capO$-algebras that is a $\TQ^\capA$-equivalence.
\item[(b)] The map $p$ has the right lifting property with respect to every subcell $\capO$-algebra inclusion $A\subset B$ that is a $\TQ^\capA$-equivalence.
\end{itemize}
\end{prop}

\begin{proof}
Since every subcell $\capO$-algebra inclusion $A\subset B$ is a strong cofibration, the implication $\text{(a)}\Rightarrow\text{(b)}$ is immediate. Conversely, suppose $p$ has the right lifting property with respect to every subcell $\capO$-algebra inclusion that is a $\TQ^\capA$-equivalence. Let $\function{i}{A}{B}$ be a strong cofibration of $\capO$-algebras that is a $\TQ^\capA$-equivalence and consider any solid commutative diagram of the form
\begin{align*}
\xymatrix{
  A\ar[d]_-{i}\ar[r]^-{g} & X\ar[d]^-{p}\\
  B\ar[r]_-{h}\ar@{.>}[ur]^{\xi} & Y
}
\end{align*}
in $\AlgO$. We want to verify that a lift $\xi$ exists. The first step is to get subcell $\capO$-algebras into the picture. Running the small object argument with respect to the generating cofibrations in $\AlgO$, we first functorially factor the map ${*}\rightarrow A$ as a cofibration followed by an acyclic fibration
$
  {*}\rightarrow A'\xrightarrow{a} A
$, and then we functorially factor the composite map $A'\rightarrow A\rightarrow B$ as a cofibration followed by an acyclic fibration
$
  A'\xrightarrow{i'} B'\xrightarrow{b} B
$. Putting it all together, we get a commutative diagram of the form
\begin{align*}
\xymatrix{
  A'\ar[d]_-{i'}\ar[r]^-{a} & A\ar[d]^-{i}\ar[r]^-{g} & X\ar[d]^-{p}\\
  B'\ar[r]^-{b} & B\ar[r]^-{h} & Y
}
\end{align*}
where $i'$ is a subcell $\capO$-algebra inclusion, by construction. Furthermore, since $i$ is a $\TQ^\capA$-equivalence and $a,b$ are weak equivalences, it follows that $i'$ is a $\TQ^\capA$-equivalence. Denote by $M$ the pushout of the upper left-hand corner maps $i'$ and $a$, and consider the induced maps $c,d,\alpha$ of the form 

\begin{align*}
\xymatrix{
  A'\ar[dd]_-{i'}\ar[rr]^-{a} && 
  A\ar[dd]^(0.6){i}\ar[rr]^-{g}\ar@/^1.0ex/[dl]_-{d} && 
  X\ar[dd]^-{p}\\
  & M\ar@{.>}^-{\alpha}[dr]\ar@{.>}@/_1.0ex/[urrr]_-{\xi'}\\
  B'\ar@/_1.0ex/[ur]^-{c}\ar[rr]^-{b} && 
  B\ar[rr]^-{h}\ar@{.>}@/_1.0ex/[uurr]_-{\xi} && 
  Y
}
\end{align*}
Since $B',A',A$ are cofibrant and $i'$ is a cofibration, we know that $M$ is a homotopy pushout (Proposition \ref{prop:certain_pushouts_are_homotopy_pushouts}); in particular, since $a$ is a weak equivalence, it follows that $c$ is a weak equivalence. Since $c,b$ are weak equivalences, we know that $\alpha$ is a weak equivalence. By assumption, $p$ has the right lifting property with respect to $i'$, and hence with respect to its pushout $d$. In particular, a lift $\xi'$ exists such that $\xi' d=g$ and $p\xi'=h\alpha$. It turns out this is enough to conclude that a lift $\xi$ exists such that $\xi i=g$ and $p\xi=h$. Here is why: Consider the factorization category $\AlgO(pg)$ (Proposition \ref{prop:factorization_category_of_a_map}) of the map $pg$, together with the objects
\begin{align*}
  \mathbf{B}:\ A\xrightarrow{i} B\xrightarrow{h} Y,\quad
  \mathbf{X}:\ A\xrightarrow{g} X\xrightarrow{p} Y,\quad
  \mathbf{M}:\ A\xrightarrow{d} M\xrightarrow{h\alpha} Y
\end{align*}
Note that giving the desired lift $\xi$ is the same as giving a map of the form
\begin{align*}
\xymatrix{
  \mathbf{X}: & 
  A\ar[r] & 
  X\ar[r] & 
  Y\\
  \mathbf{B}:\ar@<+0.5ex>@{.>}[u]_(0.47){\xi} & 
  A\ar@{=}[u]\ar[r] & 
  B\ar@{.>}[u]_(0.47){\xi}\ar[r] & 
  Y\ar@{=}[u]}
\end{align*}
in $\AlgO(pg)$. Also, we know from above that a lift $\xi'$ exists; i.e., we have shown there is a map of the form
\begin{align*}
\xymatrix{
  \mathbf{X}: & 
  A\ar[r] & 
  X\ar[r] & 
  Y\\
  \mathbf{M}:\ar@<+0.5ex>@{.>}[u]_(0.47){\xi'} & 
  A\ar@{=}[u]\ar[r] & 
  M\ar@{.>}[u]_(0.47){\xi'}\ar[r] & 
  Y\ar@{=}[u]
}
\end{align*}
in $\AlgO(pg)$. We also know from above that the map $\alpha$ is a weak equivalence, and hence we have a weak equivalence of the form
\begin{align*}
\xymatrix{
  \mathbf{M}:\ar@<-0.5ex>@{.>}[d]^-{\alpha}_-{\wequiv} & 
  A\ar@{=}[d]\ar[r] & 
  M\ar@{.>}[d]^-{\alpha}_-{\wequiv}\ar[r] & 
  Y\ar@{=}[d]\\
  \mathbf{B}: & 
  A\ar[r] & 
  B\ar[r] & 
  Y
}
\end{align*}
in $\AlgO(pg)$. Since $i,d$ are cofibrations, we know that $\mathbf{B},\mathbf{M}$ are cofibrant in $\AlgO(pg)$, and since $p$ is a fibration, we know that $\mathbf{X}$ is fibrant in $\AlgO(pg)$ (Proposition \ref{prop:factorization_category_of_a_map}). It follows that the weak equivalence $\function{\alpha}{\mathbf{M}}{\mathbf{B}}$ induces an isomorphism
\begin{align*}
  [\mathbf{M},\mathbf{X}]\xleftarrow{\Iso}[\mathbf{B},\mathbf{X}]
\end{align*}
on homotopy classes of maps in $\AlgO(pg)$, and since the left-hand side is non-empty, it follows that the right-hand side is also non-empty; in other words, there exists a map $[\xi]\in[\mathbf{B},\mathbf{X}]$. Hence we have verified there exists a map of the form $\function{\xi}{\mathbf{B}}{\mathbf{X}}$ in $\AlgO(pg)$; in other words, we have shown that the desired lift $\xi$ exists. This completes the proof of the implication $(b)\Rightarrow (a)$.
\end{proof}

\begin{prop}[Detecting $\TQ^\capA$-local $\capO$-algebras: Part 2]
\label{prop:detecting_TQ_local_O_algebras_part_2}
Let $X$ be a fibrant $\capO$-algebra. Then $X$ is $\TQ^\capA$-local if and only if $X\rightarrow *$ satisfies the right lifting property with respect to every subcell $\capO$-algebra inclusion $A\subset B$ that is a $\TQ^\capA$-equivalence.
\end{prop}

\begin{proof}
This follows immediately from Proposition \ref{prop:subcell_lifting_property}.
\end{proof}

\section{$\TQ^\capA$-local homotopy theory}

The purpose of this section is to establish a version of Proposition \ref{prop:subcell_lifting_property} (see Proposition \ref{prop:bounded_subcell_lifting_property}), and hence a corresponding version of Proposition \ref{prop:detecting_TQ_local_O_algebras_part_2} (see Proposition \ref{prop:detecting_TQ_local_O_algebras_part_3}), that includes a bound on how many cells $B$ has. Once this is accomplished, we can run the small object argument to prove the key factorization property (Proposition \ref{prop:key_factorization_needed_for_TQ_local_homotopy_theory}) needed  to establish the associated $\TQ^\capA$-local homotopy theory on $\capO$-algebras (Theorem \ref{thm:TQ_local_homotopy_theory}) and to construct the associated $\TQ^\capA$-localization functor on cofibrant $\capO$-algebras as a weak $\TQ^\capA$-fibrant (Definition \ref{defn:weak_TQ_fibrant}) replacement functor. Our argument can be thought of as an $\capO$-algebra analog of the bounded cofibration property in Bousfield \cite[11.2]{Bousfield_localization_spaces}, Goerss-Jardine \cite[X.2.13]{Goerss_Jardine}, and Jardine \cite[5.2]{Jardine_local_homotopy_theory}, mixed together with the subcell inclusion ideas in Hirschhorn \cite[2.3.7]{Hirschhorn}.

\begin{prop}
\label{prop:strong_cofibrations_give_les_in_TQ_homology_groups}
Let $\function{i}{A}{B}$ be a strong cofibration and consider the pushout diagram of the form
\begin{align}
\label{eq:quotient_of_a_strong_cofibration}
\xymatrix{
  A\ar[d]\ar[r]^-{i} & B\ar[d]\\
  {*}\ar[r] & B\sslash A
}
\end{align}
in $\AlgO$. Then there is an associated cofibration sequence of the form
\begin{align*}
  \TQ^\capA(A)\rightarrow
  \TQ^\capA(B)\rightarrow
  \TQ^\capA(B\sslash A)
\end{align*}
in $\Mod_\capA$ and corresponding long exact sequence of abelian groups of the form
\begin{align}
\label{eq:long_exact_sequence_in_TQ_star}
  \cdots\TQ^\capA_{s+1}(B\sslash A)\rightarrow
  \TQ^\capA_{s}(A)\rightarrow
  \TQ^\capA_{s}(B)\rightarrow
  \TQ^\capA_{s}(B\sslash A)\rightarrow
  \TQ^\capA_{s-1}(A)\rightarrow\cdots
\end{align}
where $\TQ^\capA_s(X):=\pi_s\TQ^\capA(X)$ denotes the $s$-th $\TQ^\capA$-homology group of an $\capO$-algebra $X$ and $\pi_*$ denotes the derived (or true) homotopy groups of a symmetric spectrum \cite{Schwede_book_project, Schwede_homotopy_groups}.
\end{prop}

\begin{proof}
This is because $Q$ is a left Quillen functor and hence preserves cofibrations and pushout diagrams.
\end{proof}

\begin{defn}
\label{defn:choice_of_large_enough_regular_cardinal}
Let $\kappa$ be a large enough (infinite) regular cardinal such that 
\begin{align*}
  \kappa > \bigl|\oplus_{s,m,k}\oplus_{H}\TQ^\capA_s\bigl(\capO\circ(D_{m}^{H,k}/S_{m}^{H,k})\bigr)\bigr|
\end{align*}
where the first direct sum is indexed over all $s\in\ZZ,\ m\geq 1,\ k\geq 0$ and the second direct sum is indexed over all subgroups $H\subset \Sigma_m$.
\end{defn}

\begin{rem}
\label{rem:significance_of_kappa}
The significance of this choice of regular cardinal $\kappa$ arises from the cofiber sequence of the form
\begin{align*}
  \TQ^\capA(Z_n)\rightarrow \TQ^\capA(Z_{n+1})
  \rightarrow\coprod_{i\in I_n}\TQ^\capA\bigl(\capO\circ(D_{m_i}^{H_i,k_i}/S_{m_i}^{H_i,k_i})\bigr)
\end{align*}
in $\Mod_\capA$ associated to the pushout diagram \eqref{eq:presentation_of_relative_cell_O_algebra}.
\end{rem}

\begin{prop}
Let $Z$ be a cell $\capO$-algebra with less than $\kappa$ cells (with respect to a choice of presentation $*\rightarrow Z$). Then 
\begin{align*}
  \bigl|\oplus_s\TQ^\capA_s(Z)\bigr| < \kappa
\end{align*}
where the direct sum is indexed over all $s\in\ZZ$.
\end{prop}

\begin{proof}
Using the presentation notation in \eqref{eq:presentation_of_relative_cell_O_algebra} (taking $W=*$), this follows from Remark \ref{rem:significance_of_kappa}, together with  Proposition \ref{prop:strong_cofibrations_give_les_in_TQ_homology_groups}, by induction on $n$. In more detail: Since $Z_0=*$ we know that $|\oplus_s\TQ^\capA_s(Z_0)|<\kappa$. Let $n\geq 0$ and assume that 
\begin{align}
\label{eq:cardinality_estimate}
  \bigl|\oplus_s\TQ^\capA_s(Z_n)\bigr|<\kappa
\end{align} 
We want to show that $\bigl|\oplus_s\TQ^\capA_s(Z_{n+1})\bigr|<\kappa$. Consider the long exact sequence in $\TQ^\capA$-homology groups of the form
\begin{align}
  \dots\rightarrow
  \TQ^\capA_{s}(Z_n)\rightarrow
  \TQ^\capA_{s}(Z_{n+1})\rightarrow
  \bigoplus_{i\in I_n}\TQ^\capA_s\bigl(\capO\circ(D_{m_i}^{H_i,k_i}/S_{m_i}^{H_i,k_i})\bigr)\rightarrow\dots
\end{align}
associated to the cofiber sequence in Remark \ref{rem:significance_of_kappa}. It follows easily that
\begin{align*}
  \bigl|\TQ^\capA_s(Z_{n+1})\bigr|\leq
  \bigl|\TQ^\capA_s(Z_n)\oplus
  \bigoplus_{i\in I_n}\TQ^\capA_s\bigl(\capO\circ(D_{m_i}^{H_i,k_i}/S_{m_i}^{H_i,k_i})\bigr)
  \bigr| <\kappa
\end{align*}
and hence $\bigl|\oplus_s\TQ^\capA_s(Z_{n+1})\bigr|<\kappa$. Hence we have verified, by induction on $n$, that \eqref{eq:cardinality_estimate} is true for every $n\geq 0$; noting that $Z\Iso Z_\infty=\colim_n Z_n$ (by definition) completes the proof.
\end{proof}

\begin{prop}[Bounded subcell property]
\label{prop:bounded_subcell_property}
Let $M$ be a cell $\capO$-algebra and $L\subset M$ a subcell $\capO$-algebra. If $L\neq M$ and $L\subset M$ is a $\TQ^\capA$-equivalence, then there exists $A\subset M$ subcell $\capO$-algebra such that
\begin{itemize}
\item[(i)] $A$ has less than $\kappa$ cells 
\item[(ii)] $A\not\subset L$
\item[(iii)] $L\subset L\cup A$ is a $\TQ^\capA$-equivalence
\end{itemize}
\end{prop}

\begin{proof}
The main idea is to develop a $\TQ^\capA$-homology analog for $\capO$-algebras of the closely related argument in Bousfield's localization of spaces work \cite{Bousfield_localization_spaces}; we have benefitted from the subsequent elaboration in Goerss-Jardine \cite[X.3]{Goerss_Jardine}. We are effectively replacing arguments in terms of adding on non-degenerate simplices with arguments in terms of adding on subcell $\capO$-algebras; this idea to work with cellular structures appears in Hirschhorn \cite{Hirschhorn} assuming left properness; however, the techniques can be made to work without the left properness assumption as indicated below.

To start, choose any $A_0\subset M$ subcell $\capO$-algebra such that
\begin{itemize}
\item[(i)] $A_0$ has less than $\kappa$ cells
\item[(ii)] $A_0\not\subset L$
\end{itemize}
Here is the main idea, which is essentially a small object argument idea: We would like $L\subset L\cup A_0$ to be a $\TQ^\capA$-equivalence (i.e., we would like $\TQ^\capA_*(L\cup A_0/\!\!/L)=0$), but it might not be. So we do the next best thing. We build $A_1\supset A_0$ such that when we consider the following pushout diagrams in $\AlgO$
\begin{align*}
\xymatrix{
  L\ar[d]\ar[r] & L\cup A_0\ar[d]\ar[r] & L\cup A_1\ar[d]\\
  {*}\ar[r] & L\cup A_0\sslash L\ar[r]^{(\#)} & L\cup A_1\sslash L
}
\end{align*}
which are also homotopy pushout diagrams in $\AlgO$, the map $(\#)$ induces
\begin{align}
\label{eq:map_induced_from_A0_to_A1}
  \TQ^\capA_*(L\cup A_0\sslash L)\rightarrow\TQ^\capA_*(L\cup A_1\sslash L)
\end{align}
the zero map; in other words, we construct $A_1$ by killing off elements in the $\TQ^\capA$-homology groups $\TQ^\capA_*(L\cup A_0\sslash L)$ by attaching subcell $\capO$-algebras to $A_0$, but in a controlled manner. Since $L\cup A_0\subset M$ is a subcell $\capO$-algebra, it follows that $M$ is weakly equivalent to the filtered homotopy colimit
\begin{align*}
  M\Iso\colim_{F_i\subset M}(L\cup A_0\cup F_i)\wequiv \hocolim_{F_i\subset M}(L\cup A_0\cup F_i)
\end{align*}
indexed over all finite $F_i\subset M$ subcell $\capO$-algebras and hence
\begin{align*}
  0=\TQ^\capA_*(M\sslash L)\Iso\colim_{F_i\subset M}\TQ^\capA_*(L\cup A_0\cup F_i\sslash L)
\end{align*}
where the left-hand side is trivial by assumption. Hence for each $0\neq x\in\TQ^\capA_*(L\cup A_0\sslash L)$ there exists a finite $F_x\subset M$ subcell $\capO$-algebra such that the induced map
\begin{align*}
  \TQ^\capA_*(L\cup A_0\sslash L)\rightarrow
  \TQ^\capA_*(L\cup A_0\cup F_x\sslash L)
\end{align*}
sends $x$ to zero. Define $A_1:=(A_0\cup\cup_{x\neq 0}F_x)\subset M$ subcell $\capO$-algebra. By construction the induced map \eqref{eq:map_induced_from_A0_to_A1} on $\TQ^\capA$-homology groups is the zero map. Furthermore, the pushout diagram in $\AlgO$
\begin{align*}
\xymatrix{
  L\cap A_0\ar[d]\ar[r] & L\ar[d]\\
  A_0\ar[r] & L\cup A_0
}
\end{align*}
implies that $L\cup A_0\sslash L\Iso A_0\sslash L\cap A_0$, hence from the cofiber sequence of the form
\begin{align*}
  L\cap A_0\rightarrow A_0\rightarrow L\cup A_0\sslash L
\end{align*}
in $\AlgO$ and its associated long exact sequence in $\TQ^\capA_*$ it follows that $A_1\subset M$ subcell $\capO$-algebra satisfies
\begin{itemize}
\item[(i)] $A_1$ has less than $\kappa$ cells
\item[(ii)] $A_1\not\subset L$
\end{itemize}

Now we repeat the main idea above, but replacing $A_0$ with $A_1$: We would like $L\subset L\cup A_1$ to be a $\TQ^\capA$-equivalence (i.e., we would like $\TQ^\capA_*(L\cup A_1/\!\!/L)=0$), but it might not be. So we do the next best thing. We build $A_2\supset A_1$ such that the induced map $\TQ^\capA_*(L\cup A_1\sslash L)\rightarrow\TQ^\capA_*(L\cup A_2\sslash L)$ is zero by attaching subcell $\capO$-algebras to $A_1$, but in a controlled manner, \dots, and so on: By induction we construct, exactly as above, a sequence of subcell $\capO$-algebras
\begin{align}
\label{eq:sequence_of_subcell_O_algebras}
  A_0\subset A_1 \subset \dots \subset A_n\subset A_{n+1}\subset \dots
\end{align}
satisfying ($n\geq 0$)
\begin{itemize}
\item[(i)] $A_n$ has less than $\kappa$ cells
\item[(ii)] $A_n\not\subset L$
\item[(iii)] $\TQ^\capA_*(L\cup A_n\sslash L)\rightarrow\TQ^\capA_*(L\cup A_{n+1}\sslash L)$ is the zero map
\end{itemize}
Define $A:=\cup_n A_n$. Let's verify that $L\subset L\cup A$ is a $\TQ^\capA$-equivalence; this is the same as checking that $\TQ^\capA_*(L\cup A\sslash L)=0$. Since \eqref{eq:sequence_of_subcell_O_algebras} is a sequence of subcell $\capO$-algebras, it follows that $L\cup A$ is weakly equivalent to the filtered homotopy colimit
\begin{align*}
  L\cup A\Iso\colim_n(L\cup A_n)\wequiv \hocolim_n(L\cup A_n)
\end{align*}
and hence
\begin{align*}
  \TQ^\capA_*(L\cup A\sslash L)\Iso\colim_n\TQ^\capA_*(L\cup A_n\sslash L)
\end{align*}
In particular, each $x\in\TQ^\capA_*(L\cup A\sslash L)$ is represented by an element in $\TQ^\capA_*(L\cup A_n\sslash L)$ for some $n$, and hence it is in the image of the composite map
\begin{align*}
  \TQ^\capA_*(L\cup A_n\sslash L)\rightarrow
  \TQ^\capA_*(L\cup A_{n+1}\sslash L)\rightarrow
  \TQ^\capA_*(L\cup A\sslash L)
\end{align*}
Since the left-hand map is the zero map by construction, this verifies that $x=0$. Hence we have verified $L\subset L\cup A$ is a $\TQ^\capA$-equivalence, which completes the proof.
\end{proof}

The following is closely related to \cite[11.3]{Bousfield_localization_spaces}, \cite[X.2.14]{Goerss_Jardine}, and \cite[5.4]{Jardine_local_homotopy_theory}, together with the subcell ideas in \cite[2.3.8]{Hirschhorn}.

\begin{prop}[Bounded subcell lifting property]
\label{prop:bounded_subcell_lifting_property}
Let $\function{p}{X}{Y}$ be a fibration of $\capO$-algebras. Then the following are equivalent:
\begin{itemize}
\item[(a)] the map $p$ has the right lifting property with respect to every strong cofibration $A\rightarrow B$ of $\capO$-algebras that is a $\TQ^\capA$-equivalence.
\item[(b)] the map $p$ has the right lifting property with respect to every subcell $\capO$-algebra inclusion $A\subset B$ that is a $\TQ^\capA$-equivalence and such that $B$ has less than $\kappa$ cells (Definition \ref{defn:choice_of_large_enough_regular_cardinal}).
\end{itemize}
\end{prop}

\begin{proof}
The implication $\text{(a)}\Rightarrow\text{(b)}$ is immediate. Conversely, suppose $p$ has the right lifting property with respect to every subcell $\capO$-algebra inclusion $A\subset B$ that is a $\TQ^\capA$-equivalence and such that $B$ has less than $\kappa$ cells. We want to verify that $p$ satisfies the lifting conditions in (a); by the subcell lifting property, it suffices to verify that $p$ satisfies the lifting conditions in Proposition \ref{prop:subcell_lifting_property}(b). Let  $A\subset B$ be a subcell $\capO$-algebra inclusion that is a $\TQ^\capA$-equivalence and consider any left-hand solid commutative diagram of the form
\begin{align}
\label{eq:lifts_and_partial_lifts}
\xymatrix{
  A\ar[d]_-{\subset}\ar[r]^-{g} & X\ar[d]^-{p}\\
  B\ar[r]_-{h}\ar@{.>}[ur]^{\xi} & Y
}\quad\quad
\xymatrix{
  A\ar[d]_-{\subset}\ar[rr]^-{g} && X\ar[d]^-{p}\\
  A_s\ar[r]_-{\subset}\ar@{.>}[urr]^{\xi_s} &
  B\ar[r]_-{h} & Y
}
\end{align}
in $\AlgO$. We want to verify that a lift $\xi$ exists. The idea is to use a Zorn's lemma argument on an appropriate poset $\Omega$ of partial lifts, together with Proposition \ref{prop:bounded_subcell_property}, following closely \cite[X.2.14]{Goerss_Jardine} and \cite[2.3.8]{Hirschhorn}. Denote by $\Omega$ the poset of all pairs $(A_s,\xi_s)$ such that (i) $A_s\subset B$ is a subcell $\capO$-algebra inclusion that is a $\TQ^\capA$-equivalence and (ii) $\function{\xi_s}{A_s}{X}$ is a map in $\AlgO$ that makes the right-hand diagram in \eqref{eq:lifts_and_partial_lifts} commute (i.e., $\xi_s|A=g$ and $p\xi_s=h|A_s$), where $\Omega$ is ordered by the following relation: $(A_s,\xi_s)\leq (A_t,\xi_t)$ if $A_s\subset A_t$ is a subcell $\capO$-algebra inclusion and $\xi_t|A_s=\xi_s$. Then by Zorn's lemma, this set $\Omega$ has a maximal element $(A_m,\xi_m)$.

We want to show that $A_m=B$. Suppose not. Then $A_m\neq B$ and $A_m\subset B$ is a $\TQ^\capA$-equivalence, hence by the bounded subcell property (Proposition \ref{prop:bounded_subcell_property}) there exists $K\subset B$ subcell $\capO$-algebra such that 
\begin{itemize}
\item[(i)] $K$ has less than $\kappa$ cells 
\item[(ii)] $K\not\subset A_m$
\item[(iii)] $A_m\subset A_m\cup K$ is a $\TQ^\capA$-equivalence
\end{itemize}
We have a pushout diagram of the left-hand form
\begin{align*}
\xymatrix{
  A_m\cap K\ar[d]\ar[r] & A_m\ar[d]\\
  K\ar[r] & A_m\cup K
}\quad\quad
\xymatrix{
A_m\cap K\ar[d]\ar[r] & A_m\ar[r]^-{\xi_m} & X\ar[d]^-{p}\\
K\ar[r]\ar@{.>}[urr]^(0.4){\xi} & B\ar[r]_-{h} & Y
}
\end{align*}
in $\AlgO$ where the indicated maps are inclusions, and by assumption on $p$, the right-hand solid commutative diagram in $\AlgO$ has a lift $\xi$. It follows that the induced map $\xi_m\cup\xi$ makes the following diagram
\begin{align*}
\xymatrix{
  A\ar[d]\ar[rrr]^-{g} &&& X\ar[d]^-{p}\\
  A_m\ar[r]\ar@{.>}[urrr]^-(0.4){\xi_m} & 
  A_m\cup K\ar[r]\ar@{.>}@/_0.5pc/[urr]_-(0.7){\xi_m\cup\xi} & B\ar[r]_-{h} & Y
}
\end{align*}
in $\AlgO$ commute, where the unlabeled arrows are the natural inclusions. In particular, since $K\not\subset A_m$, then $A_m\neq A_m\cup K$, and hence we have constructed an element $(A_m\cup K,\xi_m\cup\xi)$ of the set $\Omega$ that is strictly greater than the maximal element $(A_m,\xi_m)$, which is a contradiction. Therefore $A_m=B$ and the desired lift $\xi=\xi_m$ exists, which completes the proof.
\end{proof}

\begin{prop}[Detecting $\TQ^\capA$-local $\capO$-algebras: Part 3]
\label{prop:detecting_TQ_local_O_algebras_part_3}
Let $X$ be a fibrant $\capO$-algebra. Then $X$ is $\TQ^\capA$-local if and only if $X\rightarrow *$ satisfies the right lifting property with respect to every subcell $\capO$-algebra inclusion $A\subset B$ that is a $\TQ^\capA$-equivalence and such that $B$ has less than $\kappa$ cells (Definition \ref{defn:choice_of_large_enough_regular_cardinal}).
\end{prop}

\begin{proof}
This follows immediately from Proposition \ref{prop:bounded_subcell_lifting_property}.
\end{proof}

\begin{prop}
\label{prop:closure_under_retracts}
If $f$ is a retract of $g$ and $g$ is a $\TQ^\capA$-acyclic strong cofibration, then so is $f$.
\end{prop}

\begin{proof}
This is because strong cofibrations and weak equivalences are closed under retracts and $Q$ is a left Quillen functor.
\end{proof}

\begin{prop}
\label{prop:pushout_closure}
Consider any pushout diagram of the form
\begin{align}
\label{eq:pushout_diagram_for_closure_prop}
\xymatrix{
  A\ar[d]_-{i}\ar[r] & X\ar[d]^-{j}\\
  B\ar[r] & Y
}
\end{align}
in $\AlgO$. If $X$ is cofibrant and $i$ is a $\TQ^\capA$-acyclic strong cofibration, then $j$ is a $\TQ^\capA$-acyclic strong cofibration.
\end{prop}

\begin{proof}
Applying $Q$ to the diagram \eqref{eq:pushout_diagram_for_closure_prop} gives a pushout diagram of the form
\begin{align*}
\xymatrix{
  Q(A)\ar[d]_-{(*)}\ar[r] & Q(X)\ar[d]^-{(**)}\\
  Q(B)\ar[r] & Q(Y)
}
\end{align*}
in $\AlgO$. Since $(*)$ is an acyclic cofibration by assumption, it follows that $(**)$ is an acyclic cofibration, which completes the proof.
\end{proof}

\begin{prop}
\label{prop:closure_under_small_coproducts_and_transfinite_compositions}
The class of $\TQ^\capA$-acyclic strong cofibrations is (i) closed under all small coproducts and (ii) closed under all (possibly transfinite) compositions.
\end{prop}

\begin{proof}
Part (i) is because strong cofibrations are closed under all small coproducts and $Q$ is a left Quillen functor, and part (ii) is because strong cofibrations are closed under all (possibly transfinite) compositions and $Q$ is a left Quillen functor.
\end{proof}

\begin{defn}
\label{defn:generating_cofibrations_and_acyclic_cofibrations}
Denote by $I_{\TQ^\capA}$ the set of generating cofibrations in $\AlgO$ and by $J_{\TQ^\capA}$ the set of generating acyclic cofibrations in $\AlgO$ union the set of $\TQ^\capA$-acyclic strong cofibrations consisting of one representative  of each isomorphism class of subcell $\capO$-algebra inclusions $A\subset B$ that are $\TQ^\capA$-equivalences and such that $B$ has less than $\kappa$ cells (Definition \ref{defn:choice_of_large_enough_regular_cardinal}).
\end{defn}

\begin{prop}
\label{prop:key_factorization_needed_for_TQ_local_homotopy_theory}
Any map $X\rightarrow Y$ of $\capO$-algebras with $X$ cofibrant can be factored as $X\rightarrow X'\rightarrow Y$ a $\TQ^\capA$-acyclic strong cofibration followed by a weak $\TQ^\capA$-fibration.
\end{prop}

\begin{proof}
We know by \cite[12.4]{Hirschhorn} that the set $J_{\TQ^\capA}$ permits the small object argument \cite[10.5.15]{Hirschhorn}, and running the small object argument for the map $X\rightarrow Y$ with respect to $J_{\TQ^\capA}$ produces a functorial factorization of the form
\begin{align*}
  X\xrightarrow{j}X'\xrightarrow{p} Y
\end{align*}
in $\AlgO$. We know that $j$ is a $\TQ^\capA$-acyclic strong cofibration by Propositions \ref{prop:pushout_closure} and \ref{prop:closure_under_small_coproducts_and_transfinite_compositions}. Since $J_{\TQ^\capA}$ contains the set of generating acyclic cofibrations for $\AlgO$, we know that $p$ is a fibration of $\capO$-algebras, and hence it follows from Proposition \ref{prop:bounded_subcell_lifting_property} that $p$ is a weak $\TQ^\capA$-fibration, which completes the proof.
\end{proof}

\begin{prop}
\label{prop:right_lifting_property_characterization_set_of_maps}
Suppose $\function{p}{X}{Y}$ is a map of $\capO$-algebras.
\begin{itemize}
\item[(a)] The map $p$ is a weak $\TQ^\capA$-fibration if and only if it satisfies the right lifting property with respect to the set of maps $J_{\TQ^\capA}$ (Definition \ref{defn:generating_cofibrations_and_acyclic_cofibrations}).
\item[(b)] The map $p$ is a $\TQ^\capA$-acyclic weak $\TQ^\capA$-fibration if and only if it satisfies the right lifting property with respect to the set of maps $I_{\TQ^\capA}$ (Definition \ref{defn:generating_cofibrations_and_acyclic_cofibrations}).
\end{itemize}
\end{prop}

\begin{proof}
Part (a) was verified in the proof of Proposition \ref{prop:key_factorization_needed_for_TQ_local_homotopy_theory} and part (b) is because $p$ is an acyclic fibration (Proposition \ref{prop:TQ_acyclic_TQ_fibrations}).
\end{proof}

Our main result, Theorem \ref{thm:TQ_local_homotopy_theory}, is that the $\TQ^\capA$-local homotopy theory for $\capO$-algebras (associated to the classes of maps in Definition 
\ref{defn:classes_of_maps_TQ_local_homotopy_theory}) can be established (e.g., as a semi-model structure in the sense of Goerss-Hopkins \cite{Goerss_Hopkins} and Spitzweck \cite{Spitzweck_thesis}, that is both cofibrantly generated and simplicial) by localizing with respect to a set of strong cofibrations that are $\TQ^\capA$-equivalences; see, for instance,  Mandell \cite{Mandell}, White \cite{White}, and White-Yau \cite{White_Yau} where semi-model structures naturally arise in some interesting applications. A closely related (but different) notion of semi-model structure is explored in Fresse \cite{Fresse}.

\begin{thm}[$\TQ^\capA$-local homotopy theory: Semi-model structure]
\label{thm:TQ_local_homotopy_theory}
The category $\AlgO$ with the three distinguished classes of maps (i) $\TQ^\capA$-equivalences, (ii)
weak $\TQ^\capA$-fibrations, and (iii) cofibrations, each closed under composition and containing all isomorphisms, has the structure of a semi-model category in the sense of Goerss-Hopkins \cite[1.1.6]{Goerss_Hopkins_moduli_problems}; in more detail:
\begin{itemize}
\item[(a)] The category $\AlgO$ has all small limits and colimits.
\item[(b)] $\TQ^\capA$-equivalences, 
weak $\TQ^\capA$-fibrations, and cofibrations are each closed under retracts; weak $\TQ^\capA$-fibrations and $\TQ^\capA$-acyclic weak $\TQ^\capA$-fibrations are each closed under pullbacks.
\item[(c)] If $f$ and $g$ are maps in $\AlgO$ such that $gf$ is defined and if two of the three maps $f,g,gf$ are $\TQ^\capA$-equivalences, then so is the third.
\item[(d)] Cofibrations have the left lifting property with respect to $\TQ^\capA$-acyclic weak $\TQ^\capA$-fibrations, and $\TQ^\capA$-acyclic cofibrations with  
\textbf{cofibrant domains} have the left lifting property with respect to weak $\TQ^\capA$-fibrations.
\item[(e)] Every map can be functorially factored as a cofibration followed by a $\TQ^\capA$-acyclic weak $\TQ^\capA$-fibration and every map with \textbf{cofibrant domain} can be functorially factored as a $\TQ^\capA$-acyclic cofibration followed by a weak $\TQ^\capA$-fibration.
\end{itemize}
Furthermore, this semi-model structure is cofibrantly generated in the sense of Goerss-Hopkins \cite[1.1.7]{Goerss_Hopkins_moduli_problems} with generating cofibrations the set $I_{\TQ^\capA}$ and generating $\TQ^\capA$-acyclic cofibrations the set $J_{\TQ^\capA}$ (Definition \ref{defn:generating_cofibrations_and_acyclic_cofibrations}), and it is simplicial in the sense of \cite[1.1.8]{Goerss_Hopkins_moduli_problems}.
\end{thm}

\begin{proof}
Part (a) follows from the usual model structure on $\capO$-algebras (see, for instance, \cite{Harper_Hess}). Consider part (b). It is immediate that $\TQ^\capA$-equivalences are closed under retracts (since weak equivalences are). We know that cofibrations are closed under retracts (e.g., by the usual model structure on $\capO$-algebras). Noting that any right lifting property is closed under retracts and pullbacks, together with Proposition \ref{prop:right_lifting_property_characterization_set_of_maps}, verifies part (b). Part (c) is because weak equivalences satisfy the two-out-of-three property. Part (d) follows from Proposition \ref{prop:TQ_acyclic_TQ_fibrations} and Definition \ref{defn:classes_of_maps_TQ_local_homotopy_theory}. The first factorization in part (e) follows from Proposition \ref{prop:TQ_acyclic_TQ_fibrations} by running the small object argument with respect to the set $I_{\TQ^\capA}$ and the second factorization in part (e) is Proposition \ref{prop:key_factorization_needed_for_TQ_local_homotopy_theory} (obtained by running the small object argument with respect to the set $J_{\TQ^\capA}$). This semi-model structure is cofibrantly generated in the sense of \cite[1.1.7]{Goerss_Hopkins_moduli_problems} by Proposition \ref{prop:right_lifting_property_characterization_set_of_maps} and is simplicial in the sense of \cite[1.1.8]{Goerss_Hopkins_moduli_problems} by Proposition \ref{prop:pullback_corner_map}.
\end{proof}

\begin{defn}
\label{defn:weak_TQ_fibrant}
An $\capO$-algebra $X$ is called \emph{$\TQ^\capA$-fibrant} (resp. \emph{weak $\TQ^\capA$-fibrant}) if $X\rightarrow *$ is a $\TQ^\capA$-fibration (resp. weak $\TQ^\capA$-fibration).
\end{defn}

\begin{prop}
An $\capO$-algebra $X$ is $\TQ^\capA$-local if and only if it is weak $\TQ^\capA$-fibrant.
\end{prop}

\begin{proof}
This follows from Proposition \ref{prop:detecting_TQ_local_O_algebras_part_1} and Remark \ref{rem:dropping_fibrancy_assumption}.
\end{proof}

Let $X$ be an $\capO$-algebra and run the small object argument with respect to the set $I_{\TQ^\capA}$ for the map $*\rightarrow X$; this gives a functorial factorization in $\AlgO$ as a cofibration followed by an acyclic fibration $*\rightarrow \tilde{X}\xrightarrow{\wequiv} X$; in particular, $\tilde{X}$ is cofibrant. Now run the small object argument with respect to the set $J_{\TQ^\capA}$ for the map $\tilde{X}\rightarrow *$; this gives a functorial factorization in $\AlgO$ as $\tilde{X}\rightarrow L(\tilde{X})\rightarrow *$ a $\TQ^\capA$-acyclic strong cofibration followed by a weak $\TQ^\capA$-fibration; in particular, $L(\tilde{X})$ is $\TQ^\capA$-local and the natural zigzag $X\wequiv\tilde{X}\rightarrow L(\tilde{X})$ is a $\TQ^\capA$-equivalence. Hence we have verified the following theorem.

\begin{thm}
If $X$ is an $\capO$-algebra, then (i) there is a natural zigzag of $\TQ^\capA$-equivalences of the form $X\wequiv \tilde{X}\rightarrow L_{\TQ^\capA}(\tilde{X})$ with $\TQ^\capA$-local codomain, and if furthermore $X$ is cofibrant, then (ii) there is a natural $\TQ^\capA$-equivalence of the form $X\rightarrow L_{\TQ^\capA}(X)$ with $\TQ^\capA$-local codomain.
\end{thm}

\begin{proof}
Taking $L_{\TQ^\capA}(\tilde{X}):=L(\tilde{X})$ for part (i) and $L_{\TQ^\capA}(X):=L(X)$ for part (ii) completes the proof.
\end{proof}

\bibliographystyle{plain}
\bibliography{TQLocalization}

\end{document}